\documentclass[11pt]{amsart} 

\usepackage[utf8]{inputenc} 

  \usepackage{fullpage}

\usepackage{graphicx} 

\usepackage{amsmath}
\usepackage{amssymb}
\usepackage{amstext}
\usepackage{amsthm}
\usepackage{url}
\usepackage{booktabs} 
\usepackage{array} 
\usepackage{paralist} 
\usepackage{verbatim} 
\usepackage{subfig} 

\usepackage{fancyhdr} 

\title{Parametrizing an integer linear program by an integer}
\newcommand\blfootnote[1]{%
  \begingroup
  \renewcommand\thefootnote{}\footnote{#1}%
  \addtocounter{footnote}{-1}%
  \endgroup
}
\begin{document}
\maketitle
{\centering Bobby Shen\footnote{Bobby Shen,  Department of Mathematics, Massachusetts Institute of Technology,
Cambridge, Massachusetts, USA, \url{runbobby@mit.edu}} \par}
\newtheorem{theorem}{Theorem}[section]
\newtheorem{corollary}{Corollary}[theorem]
\newtheorem{lemma}[theorem]{Lemma}
\newtheorem{proposition}[theorem]{Proposition}
\newtheorem{conjecture}[theorem]{Conjecture}
 
\theoremstyle{definition}
\newtheorem{definition}[theorem]{Definition}
 
\theoremstyle{remark}
\newtheorem*{remark}{Remark}

\blfootnote{2010 \textit{Mathematics Subject Classsifications}: 11H06, 52C07, 90C10, 90C31 \\ Keywords: integer programming, parametric integer linear programming, convex hull, quasi-polynomial}

\begin{abstract}
We consider a family of integer linear programs in which the coefficients of the constraints and objective function are polynomials of an integer parameter $t.$ For $\ell$ in $\mathbb{Z}_+,$ we define $f_\ell(t)$ to be the $\ell^{\text{th}}$ largest value of the objective function with multiplicity for the integer linear program at $t.$ We prove that for all $\ell,$ $f_\ell$ is eventually quasi-polynomial; that is, there exists $d$ and polynomials $P_0, \ldots, P_{d-1}$ such that for sufficiently large $t,$ $f_\ell(t)=P_{t \pmod{d}}(t).$ Closely related to finding the $\ell^{\text{th}}$ largest value is describing the vertices of the convex hull of the feasible set. Calegari and Walker showed that if $R(t)$ is the convex hull of $\mathbf{v_1}(t), \ldots, \mathbf{v_k}(t)$ where $\mathbf{v_i}$ is a vector whose coordinates are in $\mathbb{Q}(t)$ and of size $O(t),$ then the vertices of the convex hull of the set of lattice points in $R(t)$ has eventually quasi-polynomial structure. We prove this without the $O(t)$ assumption.
\end{abstract}
\section{Introduction}

An integer program is the optimization of a certain objective function over the integers subject to certain constraints. Often in integer programming, the constraints and objective functions are linear functions of the indeterminates. This is known as integer linear programming. 

Suppose that the indeterminates are $\mathbf{x}=(x_1, \ldots, x_n).$ A program in general form has the following structure:

1) The objective function is $ \mathbf{c}^{\intercal} \mathbf{x}$ for $\mathbf{c} \in \mathbb{Z}^n$ 

2) The constraints are $A\mathbf{x} \le \mathbf{b}$ for $A \in \mathbb{Z}^{m \times n}$ and $b \in \mathbb{Z}^m.$ (In this paper, relations between vectors are coordinate-wise.) 

Parametric Integer Linear Programming (PILP) refers to considering a family of linear integer programs parametrized by a variable $t$ i.e. the coefficients of the objective and/or constraints are functions of $t.$ The optimum value of the objective function is a function of $t$ (which we call the optimum value function), which leads us to questions about this function. Examples for the domain of $t$ are the interval $[0,1]$ or the positive integers. 
Our main question concerns parametric integer linear programs in which all coefficients are integer polynomials in $t$ and $t$ ranges over the positive integers. In this paper, $t$ is a positive integer. We prove that the optimum value function is eventually quasi-polynomial.

\begin{definition} $g$ is \textit{eventually quasi-polynomial} (EQP) if its domain is a subset of $\mathbb{Z}$ that contains all sufficiently large integers and there exists a positive integer $d$ and polynomials $P_0, \ldots, P_{d-1} \in \mathbb{R}[t]$ such that for sufficiently large $t,$ $g(t)=P_{t \pmod{d}}(t).$
\end{definition}

There are many results on EQPs ranging from the simple to the sophisticated. For example, if $P_1, \ldots, P_n$ are in $\mathbb{Z}[t],$ then $\gcd(P_1(t), \ldots, P_n(t))$ and $\text{lcm}(P_1(t), \ldots, P_n(t))$ are quasi-polynomials of $t$ (even when the domain is all integers). If $P$ is in $\mathbb{Q}(t),$ one can show that $\lfloor P(t) \rfloor$ is EQP; it is in fact implied by Theorem \ref{GF} below. An early result on quasi-polynomials was by Ehrhart in \cite{Ehr}. 

\begin{theorem}
Let $P$ be a convex polytope whose vertices are rational vectors. For all $t,$ let $f(t)$ be the number of lattice points in $tP.$ Then $f$ is quasi-polynomial (\textit{not} just EQP). \end{theorem}

In fact, Macdonald showed that $f$ is meaningful for negative integers \cite{Mac}. The set $tP$ is the real vector set (see subsection 1.1) of a PILP, so this theorem states that the size function of a certain class of PILPs is EQP. Chen, Li, and Sam showed that this is true for all PILPs \cite{CLS}. We discuss their result later.

Woods proved that a certain class of parametric lattice point-counting problems defined by linear inequalities, boolean operations, and quantifiers has EQP behavior, even with more than one parameter \cite{Woo2}. The results in our paper are not true for more than one parameter. See the remark at the end of section \ref{S4}.

A common theme among these results is that certain operations on polynomials of one or more integer parameters can form other integer-valued functions with domain $\mathbb{Z}_+$ such as a GCD function or a size function which are EQP. In some cases, the function may even be quasi-polynomial with meaningful values for negative arguments. If the function is not integer-valued, like a quotient, then we do not expect anything. Our results follow this first pattern. 

We study the $\ell^{\text{th}}$ largest value attained by the objective function with multiplicity. By this, we mean the $\ell^{\text{th}}$ largest value in the multiset obtained by evaluating the objective function on the set. We prove the following theorems, which also define the structure of a PILP in general or standard form. In these theorems, we require a certain boundedness condition which many reasonable PILPs satisfy.

%
%

\begin{theorem}[General Form] \label{GF} Let $n$ and $m$ be positive integers. Let $\mathbf{c}$ be in $\mathbb{Z}[t]^n$, $A$ be in $\mathbb{Z}[t]^{m \times n}$, and $\mathbf{b}$ be in $\mathbb{Z}[t]^m.$

For all $t,$ let $R(t):=\{\mathbf{x} \in \mathbb{R}^n \mid A(t) \mathbf{x} \le \mathbf{b}(t)\},$ the set of real vectors that satisfy all constraints except being integer vectors. Let $L(t):=R(t) \cap \mathbb{Z}^n,$ the set of lattice points in $R(t).$ Assume that $R(t)$ is bounded for all $t.$ For all positive integers $\ell,$ let $f_\ell(t)$ be the $\ell^{\text{th}}$ largest value of $\mathbf{c}^{\intercal}(t) \mathbf{x}$ for $\mathbf{x}$ in $L(t)$ or $-\infty$ if $|L(t)|<\ell$. 

Then for all $\ell,$ $f_\ell$ is eventually quasi-polynomial.
\end{theorem}


\begin{theorem}[Standard Form] \label{SF} Let $n$ and $m$ be positive integers. Let $\mathbf{c}$ be in $\mathbb{Z}[t]^n$, $A$ be in $\mathbb{Z}[t]^{m \times n}$, and $\mathbf{b}$ be in $\mathbb{Z}[t]^m.$

For all $t$, let  $R(t):=\{\mathbf{x} \in \mathbb{R}^n \mid \mathbf{x} \ge \mathbf{0} \wedge A(t) \mathbf{x} = \mathbf{b}(t)\},$ the set of real vectors that satisfy all constraints except being integer vectors. Let $L(t):=R(t) \cap \mathbb{Z}^n,$ the set of lattice points in $R(t).$ Assume that $R(t)$ is bounded for all $t.$  For all positive integers $\ell,$ let $f_\ell(t)$ be the $\ell^{\text{th}}$ largest value of $\mathbf{c}^{\intercal}(t) \mathbf{x}$ for $\mathbf{x}$ in $L(t)$ or $-\infty$ if $|L(t)|<\ell$. 

Then for all $\ell,$ $f_\ell$ is eventually quasi-polynomial.
\end{theorem}

\begin{remark}By definition, A PILP in canonical form is a PILP in general form which includes the constraints $\mathbf{x} \ge \mathbf{0}.$ The standard and canonical forms are special cases of the general form because $A(t) \mathbf{x} = \mathbf{b}(t)$ is equivalent to the conjunction of $A(t) \mathbf{x} \le \mathbf{b}(t)$ and $-A(t) \mathbf{x} \le -\mathbf{b}(t)$\end{remark}

The constant function at $-\infty$ is to be interpreted as a polynomial. The general form and standard form are closely related, as we prove in the next section. Both forms of PILPs are considered because in general, it is easier to formulate a problem as a PILP in general form, but the reduction in Section 3 is more convenient in standard form.

A result by Woods, Theorem 3.5(a), Property 3a in \cite{Woo}, can be seen as a special case of the above theorems. Woods proved that if we restrict the objective function $\mathbf{c}$ to be constant, then $f_1$ is eventually quasi-polynomial.

Chen, Li, and Sam proved (as Theorems 1.1 and 2.1 in \cite{CLS}), that for a PILP in general form, the cardinality of $L(t)$ as a function of $t$ is eventually quasi-polynomial, which is a generalization of Ehrhart's theorem. They used base $t$ representations to reduce the problem to the case when $A$ is in $\mathbb{Z}^{m \times n}$ and the coordinates of $\mathbf{b}$ have degree at most 1. We show that the same idea applies when considering the $\ell^{\text{th}}$ largest value.

It is likely that in many cases, the components of the EQP $f_\ell,$ or the polynomials $P_0, \ldots, P_{d-1}$ shown in the definition, have the same degree. However, we do not discuss this in this paper. We also do not find explicit bounds on the period of $f_\ell,$ or $d$ from the definition, or an explicit integer $N$ such that $f_\ell(t)=P_{t \pmod{d}}(t)$ for $t>N.$ 

The motivation for pursuing the results in this paper was a parametrized version of the Frobenius problem in which the arguments are integer valued polynomials. See \cite{Shen}. Theorem \ref{GF} and related ideas may be useful in studying a variety of parametric combinatorics problems whose answers are suspected to be eventually quasi-polynomial such as those in \cite{Woo}.

Closely related to the idea of finding an optimum value is finding the vertices of the convex hull of $L(t)$. This is because the optimum value of a linear objective function is attained at some vertex of the convex hull of the feasible set. Calegari and Walker showed that if the vertices of $R(t)$ are of size $O(t),$ then the vertices of the convex hull of $L(t)$ have eventually quasi-polynomial structure \cite{CW}. See theorem \ref{CW}. We show that this is true without the $O(t)$ assumption as a consequence of the base $t$ method.

\subsection{Notation}
In this paper, relations between vectors are coordinate-wise. The variables $t$ and $\ell$ are always positive integers. The phrase $t \gg 0$ means ``for sufficiently large $t.$" The constant function at $-\infty$ is to be interpreted as a polynomial. PILP means ``parametric integer linear program." EQP means ``eventually quasi-polynomial" or ``eventual quasi-polynomial." Magnitude and distance for vectors refer to the Euclidean metric. By $\ell^{\text{th}}$ largest value, we mean the $\ell^{\text{th}}$ largest value with multiplicity.

By $\mathbf{x},$ we mean the vector $(x_1, \ldots, x_n).$ A parametric inequality has the form $\mathbf{a}^{\intercal}(t) \mathbf{z} \le b(t),$ where $\mathbf{a}$ is a $\mathbb{Z}[t]$-vector of the correct dimension, $b$ is in $\mathbb{Z}[t],$ and $\mathbf{z}$ refers to the vector of indeterminates e.g. $\mathbf{x},$ and a parametric equation is the same with equality. 

We say that the PILP in general or standard form is defined by $n, m, A, \mathbf{b},$ and $\mathbf{c}$ as in Theorems \ref{GF} and \ref{SF}. We call $R(t)$ the real vector set of the PILP and $L(t)$ the lattice point set. We may also call these sets regions. We use $\{f_\ell\}_\ell$ or $\{f_\ell\}$ to denote the family of optimum value functions.

\section{Proof that Theorem \ref{SF} implies Theorem \ref{GF}}
\label{S2}

The other direction is clear because one parametric equation can be written as two parametric inequalities.

First, we prove an intermediate result, which is the EQP property for PILPs in canonical form. In this form, we require all indeterminates to be nonnegative. 

\begin{theorem}[Canonical Form] \label{CF} Let $n$ and $m$ be positive integers. Let $\mathbf{c}$ be in $\mathbb{Z}[t]^n$, $A$ be in $\mathbb{Z}[t]^{m \times n}$, and $\mathbf{b}$ be in $\mathbb{Z}[t]^m.$

For all $t,$ let $R(t):=\{\mathbf{x} \in \mathbb{R}^n \mid \mathbf{x} \ge 0 \wedge A(t) \mathbf{x} \le \mathbf{b}(t)\},$ the set of real vectors that satisfy all constraints except being integer vectors. Let $L(t):=R(t) \cap \mathbb{Z}^n,$ the set of lattice points in $R(t).$ Assume that $R(t)$ is bounded for all $t.$ For all positive integers $\ell,$ let $f_\ell(t)$ be the $\ell^{\text{th}}$ largest value of $\mathbf{c}^{\intercal}(t) \mathbf{x}$ for $\mathbf{x}$ in $L(t)$ or $-\infty$ if $|L(t)|<\ell$. 

Then for all $\ell,$ $f_\ell$ is eventually quasi-polynomial.
\end{theorem}

Assume Theorem \ref{SF}. To prove Theorem \ref{CF}, we use the standard method of introducing slack variables. Let the PILP in canonical form be given by positive integers $n$ and $m,$ $\mathbf{c}$ in $\mathbb{Z}[t]^n$, $A$ in $\mathbb{Z}[t]^{m \times n}$, and $\mathbf{b}$ in $\mathbb{Z}[t]^m$ which defines optimum value functions $\{f_\ell\},$ regions $R(t)$ and $L(t)$ which are bounded for all $t.$ Consider the PILP in standard form with indeterminates $\mathbf{y}=(y_1, \ldots, y_{n+m}).$ Let $\mathbf{y^1} = (y_1, \ldots, y_n),$ $\mathbf{y^2} = (y_{n+1}, \ldots, y_{n+m}),$ and let the constraints be $A(t) \mathbf{y^1} + \mathbf{y^2} = \mathbf{b}(t)$ and $y_i \in \mathbb{Z}_{\ge 0}$ for all $i.$ The former set of constraints can easily be written as a parametric equation of $\mathbf{y}.$ Let the objective function be $\mathbf{c}^{\intercal}(t)\mathbf{y_1},$ which can be written as a polynomial covector times $\mathbf{y}.$ Let $R^*(t)$ and $L^*(t)$ be the regions defined by the second PILP and $\{f^*_\ell\}$ be the optimum value functions.

To use Theorem \ref{SF}, we wish to show that $R^*(t)$ is bounded for each $t.$ By assumption, $R(t)$ is bounded. As a subset of $\mathbb{R}^n,$ $R(t)$ is bounded by some box $[-M, M]^n$ where $M$ is constant (for fixed $t$). Let $\mathbf{y}$ be in $R^*(t).$ The first $n$ coordinates of $y$ lie in $R(t),$ so these coordinates lie in $[-M,M].$ When $k$ is an integer in $[n+1, n+m],$ $y_k= (\mathbf{b}(t) - A(t) (y_1, \ldots, y_n))_{k-n}.$ For fixed $t,$ the entries of $\mathbf{b}(t)$ and $A(t)$ are bounded in magnitude by some constant $N.$ It follows that $|y_k| \le N+nNM$ and that $R^*(t)$ is bounded for all $t.$ By Theorem \ref{SF}, for each $\ell,$ $f^*_\ell$ is EQP. 

There is an obvious bijection between the lattice point sets: if $(a_1, \ldots, a_n)$ lies in $L(t),$ then $(a_1, \ldots, a_{n+m})$, where $a_k$ (for $n+1 \le k \le n+m$) is defined as $(\mathbf{b}(t) - A(t) (a_1, \ldots, a_n))_{k-n}$, lies in $L^*(t)$. The inverse map is just ignoring the last $m$ coordinates, and it maps $L^*(t)$ to $L(t).$ Therefore, the two sets have the same cardinality for all $t.$ Evaluating the objective function commutes with the bijection by construction, so for all $\ell$ and $t,$ $f_\ell(t)=f_\ell^*(t).$

\begin{remark} If the PILP in canonical form were instead in general form, then the proposed bijection from $L(t)$ to $L^*(t)$ would not be a bijection in general because $L(t)$ may contain a point $(a_1, \ldots, a_n)$ with negative coordinates whereas all points in $L^*(t)$ have positive coordinates. \end{remark}

Now, we prove Theorem \ref{GF}. The idea is that for a PILP in general form (with the boundedness assumption), there is a uniform polynomial bound on the coordinates of feasible points.  This is essentially a standard result about magnitudes of feasible points of bounded integer linear programs. 

\begin{proposition} \label{bounded} 
For a PILP in general form, there exists a positive integer $r$ such that for $t \gg 0$, all coordinates of any point $\mathbf{x}$ in $L(t)$ have magnitude less than $t^r.$ 
 \end{proposition}

\begin{proof}
Fix $t,$ and consider the single region $R(t)$ (forget the parametric structure). By assumption, $R(t)$ is a bounded convex polytope, so the extreme values of the coordinates occur at vertices. The region $R(t)$ is defined by finitely many bounding hyperplanes, and a vertex $v$ of $R(t)$ is the unique intersection of the bounding hyperplanes on which it lies. In general, among a set of hyperplanes in $\mathbb{R}^n$ which intersect at a single point, some $n$ of those hyperplanes intersect at a single point. Such $n$ can be found by choosing a basis for the space spanned by the dual vectors to the hyperplanes. Therefore, the vertex $v$ is the unique intersection of some $n$ bounding hyperplanes to $R(t).$

We now proceed similarly to Lemma 2.1 on page 30 of \textit{Combinatorial Optimization}\cite{PS}. The $n$ hyperplanes correspond to an $n \times n$ minor of $A(t),$ call it $D.$ Let $\alpha(t)$ be the largest magnitude of a matrix element of $A(t),$ and let $\beta(t)$ be the largest magnitude of an element of $\mathbf{b}(t).$

Each coordinate of $v$ is the sum of $n$ products of elements of $D^{-1}$ by elements of $\mathbf{b}(t).$ By the cofactor description of matrix inverses, each element of $D^{-1}$ is equal to an $(n-1) \times (n-1)$ determinant divided by a nonzero $n \times n$ determinant. Since $A(t)$ is an integer matrix, this denominator has absolute value at least $1.$ By the Leibniz formula, the numerator is the sum of $(n-1)!$ products of $(n-1)$ matrix elements of $D,$ which are in turn matrix elements of $A.$ Therefore, we have
\[ |v_i| \le n (n-1)! \alpha(t)^{n-1} \beta(t).\]

This formula is true for all $t.$ It is easy to prove that $\alpha(t)$ and $\beta(t)$ are dominated by suitably chosen polynomials of $t$, and likewise, the right hand side is dominated by a suitably chosen power of $t,$ $t^r.$
\end{proof}

Consider $Q,$ a PILP in general form. Assume Theorem \ref{SF}. Let $r$ be an integer guaranteed by Proposition \ref{bounded}. It is straightforward to construct a new PILP, $Q',$ in general such that for all $t,$ the new lattice point set is exactly the old lattice point set translated by $(+t^r, +t^r, \ldots, +t^r)^{\intercal}.$ Therefore, by adding the constraints $\mathbf{x} \ge 0$ to the new PILP, we get a PILP, $Q''$, in canonical form, and the optimum value functions are unchanged for $t \gg 0.$ By Theorem \ref{CF}, the optimum value functions of $Q''$ are EQP. The same is true for $Q'.$ The optimum value functions of $Q'$ and $Q$ differ by $\mathbf{c}(t) (+t^r, +t^r, \ldots, +t^r)^{\intercal}$ or are $-\infty$ in the same places, so the optimum value functions of $Q$ are EQP as well. Therefore,  Theorem \ref{SF} implies Theorem \ref{GF}.

\section{Reduction using base $t$ representations}

In this section, we show that Theorem \ref{SF} can be reduced to the case of a PILP in canonical form in which the matrix has constant entries and the vector on the right hand side has entries which have degree at most $1$. The main idea of this reduction is to express all indeterminates $x_i$ in base $t.$ This idea is taken from Chen et. al. \cite{CLS}, in which the authors used this method to show that the cardinality of $L(t)$ is eventually quasi-polynomial. It is partially reproduced here with adjustments for considering optimum value functions. 

\begin{theorem}[Reduced canonical form] \label{RCF}
Let $n$ and $m$ be positive integers. Let $\mathbf{c}$ be in $\mathbb{Z}[t]^n$, $A$ be in $\mathbb{Z}^{m \times n}$, and $\mathbf{b}$ be in $\left(\mathbb{Z}t + \mathbb{Z}\right)^m.$

For all $t$, let $R(t):=\{\mathbf{x} \in \mathbb{R}^n \mid \mathbf{x} \ge \mathbf{0} \wedge A(t) \mathbf{x} \le \mathbf{b}(t)\},$ the set of real vectors that satisfy all constraints except being integer vectors. Let $L(t):=R(t) \cap \mathbb{Z}^n,$ the set of lattice points in $R(t).$ Assume that $R(t)$ is bounded for all $t.$ For all $\ell,$ let $f_\ell(t)$ be the $\ell^{\text{th}}$ largest value of $\mathbf{c}^{\intercal}(t) \mathbf{x}$ for $\mathbf{x}$ in $L(t)$ or $-\infty$ if $|L(t)|<\ell$. 

Then for all $\ell,$ $f_\ell$ is eventually quasi-polynomial.
\end{theorem}

We prove Theorem \ref{RCF} in the next section. The rest of this section is a proof that Theorem $\ref{RCF}$ implies Theorem $\ref{SF},$ so we are assuming the former.
 The first main part of this section recalls Lemma 3.2 of \cite{CLS}, which is the core of the base $t$ method; it decomposes PILPs in standard form into PILPs in canonical form. Chen et. al. only stated their lemma for PILPs in standard form with $m = 1,$ but it is easy to strengthen their lemma to $m>1,$ and we outline the steps needed here. The second main part of this section shows how their lemma applies to optimum value functions.

Lemma 3.2 of \cite{CLS} essentially states the following.

\begin{lemma}[Lemma 3.2 of \cite{CLS}]\label{bt0}
Let $n$ and $r$ be positive integers. Let $m = 1.$ Let $A$ be in $\mathbb{Z}[t]^{m \times n}$, and let $\mathbf{b}$ be in $\mathbb{Z}[t]^m$ with positive leading coefficient. Let $L(t)$ be the set 
\[\{\mathbf{x} \in \mathbb{Z}^n \mid 0 \le x_i < n^{r} \wedge A(t) \mathbf{x} = \mathbf{b}(t)\}.\]
Let $y_{i,j}$ for $1 \le i \le n, 0 \le j < r$ be indeterminates, and let $\mathbf{y}$ refer to the $y_{i,j}$ collectively. Define the map $\varphi_t: \{0, \ldots, t^r-1\}^n \rightarrow \{0, \ldots, t-1\}^{rn}$ by $\varphi_t(\mathbf{x}) = \mathbf{y}$ such that for $1 \le i \le n,$ 
\[ x_i = \sum_{j=0}^{r-1} y_{i, j} t^j.\]
Then there exist finitely ``PILPs''\footnote{We are ignoring the objective function and just focusing on the lattice point sets.} $Q_{\alpha}$ in reduced canonical form with following properties. Each $Q_{\alpha}$ has indeterminates $(y_{i,j}),$ includes the constraints $y_{i,j} < t,$ and has lattice point set $L_{\alpha}(t).$ For $t \gg 0,$ $\varphi_t$ is (restricts to) a bijection from $L(t)$ to the disjoint union of $L_{\alpha}(t).$
\end{lemma}

\begin{remark} In other words, $L(t)$ is the lattice point set of a PILP in standard form with $m = 1,$ and $r$ is an integer guaranteed to exist by Proposition \ref{bounded}. 
\end{remark}

\begin{remark} It is clear that $\varphi_t$ is a well-defined function and bijection from $ \{0, \ldots, t^r-1\}^n$ to $\{0, \ldots, t-1\}^{rn}.$ The inverse map $\varphi_t^{-1}$ is a bijection and also an affine transformation.
\end{remark}
%

We refer readers to their paper for the proof, along with Example 3.5 from their paper. Chen et. al. do not explicitly state that the sets $L_{\alpha}(t)$ are disjoint, but this is easily proven.\footnote{Near the end of the proof of Lemma 3.2, Chen et. al. state that for their explicit example, one of the PILPs $Q_{\alpha}$ is given by the data $(C_0, C_1) = (0,2)$ and the three equations $2 x_{10} + x_{20} = n - 5,$ etc, and the possible indexing values $\alpha$ are essentially the possible values of $(C_0, C_1).$ Suppose, for contradiction, that a point $(x_{ij})$ were in multiple sets $L_{\alpha}(n).$ Then there are multiple values of $(C_0, C_1)$ that make the three equations true. Assuming that $n > 0,$ the first equation implies that there is only one possible value of $C_0,$ and then the second equation implies that there is only one possible value of $C_1,$ a contradiction. The general case is similar.}

Chen et. al. also address PILPs in standard form with $m>1$ informally in the proof of Theorem 1.1 at the end of page 9. They argue that we can iteratively apply Lemma \ref{bt0} to reduce the number of constraints which have degree greater than $1.$ Eventually, all constraints will have degree at most $1$. It may be conceptually easier to allow $m>1$ in the first place when stating and proving Lemma \ref{bt0}. 

\begin{proposition}\label{bt1} Lemma \ref{bt0} is also true for $m>1.$ \end{proposition}

\begin{proof}[Proof (sketch)] We follow the argument and notation in Lemma 3.2 of \cite{CLS}, which is quite different from this paper's. In the $m=1$ case, Equation $(3.1)$ of \cite{CLS} reads
\[ \sum_{i=1}^k \left( \left( \sum_{ \ell = 0}^{d_i} a_{i \ell} n^{\ell} \right) \left( x_{i, N} n^N + \cdots + x_{i, 0} \right) \right) = m_N(n) n^N + \cdots + m_0(n).\]
In the $m>1$ case, we have $m$ equations of a similar form.

In the $m=1$ case, the constant term of $(3.1)$ states that
\[a_{10} x_{1,0} + \cdots + a_{k0} x_{k,0} = m_0(n) \pmod{n}.\]
As integers, the difference between the two sides is a multiple of $n.$ There are finitely many possible values for this multiple of $n$ (which don't depend on $n$). To be precise, this multiple of $n$ is between
$-n \sum_i |a_{i,0}| \text{ and } n \sum_i |a_{i,0}|.$
In the $m>1$ case, we consider the constant terms of $(3.1)$ simultaneously. We have $m$ equations mod $n.$ Over integers, each equation has a difference in the two sides which is a multiple of $n,$ and each equation has finitely many possible differences (not depending on $n$). Since $m$ is finite, there are finitely many possible $m$-tuples of differences. 

Next, we look at the degree $1$ terms in $(3.1).$ In the $m=1$ case, there are finitely many possible differences, and in the general case, there are still finitely many possible $m$-tuples of differences. This is true for all degrees up to $d.$ Overall, there are finitely many possible $(md)$-tuples of differences. For each tuple, the set of vectors $(x_{i,j})$ with those differences is exactly described by the lattice point set of a certain PILP in reduced canonical form. Disjointness can be proven as in the $m=1$ case.
\end{proof}

So far, we have only shown that a PILP in standard form is in bijection with the disjoint union of finitely many PILPs in reduced canonical form. This is enough for counting points, but we need a few more observations for optimum value functions. 

We now prove Theorem \ref{SF} assuming Theorem \ref{RCF}. Let $Q$ be a PILP in standard form given by $n, m, \mathbf{c}, A, \mathbf{b}$ with lattice point set $L(t).$ By Proposition \ref{bounded}, there exists $r>0$ such that for $t \gg 0,$ all coordinates of all points in $L(t)$ are in the range $[0, t^r).$ We want to prove that $\{f_{\ell}\},$ the optimum value functions for $Q,$ are EQP. Therefore, we can ignore the finitely many values of $t$ for which $L(t)$ is not contained in $[0,t^r)^n.$ Then, adding the constraints $x_i < t^r$ does not change $L(t).$ 

To apply Proposition \ref{bt1}, we want that all coordinates of $\mathbf{b}$ have positive leading coefficient. Since we can multiply constraints $A(t) \mathbf{x} = \mathbf{b}(t)$ by $\pm 1$ and get the same lattice point set $L(t),$ we can deal with all nonzero elements of $\mathbf{b}(t).$ Since we can add constraints among these $m$ constraints together in an ``invertible'' way, we can deal with all cases except possibly $\mathbf{b}(t) = \mathbf{0}.$ However, in this case, $L(t)$ is closed under multiplication by positive integers and also bounded, meaning that $L(t)$ is either $\{\mathbf{0}\}$ or the empty set,  and the conclusion of Theorem \ref{SF} is easily verified.

 Let $(Q_{\alpha})_{\alpha \in S}$ be the finitely many PILPs in reduced canonical form as guaranteed by Proposition \ref{bt1}. Then $\varphi_t$ is a bijection from $L(t)$ to $\sqcup_\alpha L_\alpha(t).$ Actually, we haven't yet defined objective functions for the $Q_{\alpha}.$ Let each $Q_{\alpha}$ have objective function $c$ which maps $\mathbf{y} = (y_{i,j})$ to $\mathbf{c}^\intercal(t) \varphi_t^{-1}(\mathbf{y}).$ One can check that
\[ \varphi_t^{-1}((y_{i,j})_{1 \le i \le n, 0 \le j < r}) = \left( \sum_{j = 0}^r y_{1, j} t^j, \ldots, \sum_{j=0}^r y_{n, j}t^j \right).\]
Therefore objective function $c$ can be written as a polynomial covector times $\mathbf{y}$ (thus in the form required by Theorem \ref{RCF}), although we won't do so explicitly.  Note that all $Q_{\alpha}$ have the same objective function. By construction, the objective functions of $Q_{\alpha}$ and $Q$ commute with the bijection $\varphi_t$.  Let $f_{\ell, \alpha}$ be the $\ell^{\text{th}}$ optimum value function for the PILP $Q_{\alpha}$ and this objective function. By Theorem \ref{RCF}, each function $f_{\ell, \alpha}$ is EQP. The following lemma relates the functions $\{f_{\ell, \alpha}\}$ to $\{f_{\ell}\}.$

\begin{lemma}\label{fla}Let $t$ and $\ell$ be positive integers. Then $f_{\ell}(t)$ equals the $\ell^{\text{th}}$ largest value among the multiset
\begin{displaymath} \{f_{m, \alpha}(t) \vert 1 \le m \le \ell, \alpha \in S \}.\end{displaymath}
\end{lemma}

\begin{proof} Fix $t$ and $\ell.$ By commutativity of the bijection and the objective function, $f_{\ell}(t)$ equals the $\ell^{\text{th}}$ largest value of the objective function $c$ in the set $\sqcup_{\alpha} L_{\alpha}(t).$ Thus $f_{\ell}(t)$ equals the $\ell^{\text{th}}$ largest value of a certain (disjoint) union of multisets:
\[ \sqcup_{\alpha} \{ \mathbf{c}^\intercal(t) \mathbf{x} \mid \mathbf{x} \in L_{\alpha}(t) \}.\]
 This uses the fact that the sets $L_{\alpha}(t)$ for $\alpha \in S$ are disjoint.  It is a general statement that the $\ell^{\text{th}}$ largest value in a union of multisets can be found by ignoring all but the $\ell$ largest values in each multiset; it suffices to look at the multiset of the $\ell$ largest elements in each multiset. Our convention is that the $\ell^{\text{th}}$ largest value in a multiset with less than $\ell$ elements is $-\infty.$

 By definition, the $\ell$ largest values in the multiset formed by evaluating $c$ on $L_{\alpha}(t)$ equals the multiset
$\{ f_{1, \alpha}(t), \ldots, f_{\ell, \alpha}(t) \}.$
The conclusion follows.
\end{proof}

Thus $f_{\ell}$ is pointwise related to a finite collection of EQP functions. The following more general proposition shows that $f_{\ell}$ is itself EQP.

\begin{proposition} \label{max} Let $m$ and $\ell$ be positive integers and $f_{1}, \ldots,  f_m$ be eventual quasi-polynomials with codomain $\{-\infty\} \cup \mathbb{Z}$. For all $t,$ let $f(t)$ be the $\ell^{\text{th}}$ largest value among the multiset $\{f_1(t), \ldots,$ $ f_m(t)\}.$ Then $f$ is eventually quasi-polynomial.
\end{proposition}

\begin{proof} 
Because $m$ is finite, there exists a common period for $f_1, \ldots, f_m$, or an integer $d$ such that there exist polynomials $P_{i, j}$ for all integers $i$ and $j$ with $1 \le i \le m$ and $0 \le j \le d-1$ such that for $t \gg 0 $ and all $i,$ $f_i(t)=P_{i, (t \pmod{d})}(t).$ A polynomial with range $\{-\infty\} \cup \mathbb{Z}$ is either the constant $-\infty$ or integer valued. If $f(t)$ restricted to each residue class $\pmod{d}$ is EQP, then $f$ is EQP. Since we only care about large $t$, it suffices to prove this proposition in the case when $f_1, \ldots, f_m$ are all polynomials (or the constant $-\infty$), so assume that this is the case.

If less than $\ell$ of the polynomials $f_1, \ldots, f_m$ are integer valued, then for all $t,$ the multiset $\{f_1(t), \ldots, f_m(t)\}$ contains less than $\ell$ integers, so $f(t)=-\infty.$ If at least $\ell$ of the polynomials $f_1, \ldots, f_m$ are integer valued, say $f_1$ through $f_k,$ then for all $t,$ $f(t)$ is the $\ell^{\text{th}}$ largest value among the integers $f_1(t), \ldots, f_k(t)$. It is known that there exists a permutation of the polynomials $f_1, \ldots, f_k,$ call it $f_{\sigma(1)}, \ldots, f_{\sigma(k)},$ such that for $t \gg 0,$ $f_{\sigma(1)}(t) \ge \cdots \ge f_{\sigma(k)}(t),$ so $f(t) = f_{\sigma(\ell)}(t)$ for $t \gg 0,$ as desired.
\end{proof}

Proposition \ref{max} and Lemma  \ref{fla} immediately imply that $f_\ell$ is EQP. 
This completes the proof that Theorem \ref{SF} follows from Theorem \ref{RCF}

\section{Proof of Theorem \ref{RCF}}
\label{S4}

We now allow the case $n=0$ in Theorem \ref{RCF}. We prove Theorem \ref{RCF} first in the case $n=0.$ The idea is to prove the remaining cases by strong induction on $\ell + n.$ We are introducing the $n=0$ case because our general proof involves reducing the $n = n_0+1$ case to the $n = n_0$ case, even if $n_0 = 1.$



Suppose that $n=0.$  Mathematically, this means that we have no indeterminates. We only have the zero point in our universe, and for each $t,$ $L(t)$ is either empty or is the singleton of that point. Then $f_\ell(t)=0$ if both $\ell=1$ \textit{and} $L(t)$ is nonempty, and $f_\ell(t)=-\infty$ otherwise. For any matrix $A,$ $A$ times our zero point is the zero vector, so $L(t)$ is nonempty if and only if $\mathbf{0} \le \mathbf{b}(t).$ Since the elements of $\mathbf{b}$ are polynomials, this inequality is either eventually true or eventually false (as $t$ goes to $\infty$). This corresponds to $f_1(t)=0$ eventually or $f_1(t)=-\infty$ eventually, respectively, as desired.

We now prove some propositions needed for the induction. Let $Q$ be a PILP in canonical form given by positive integers $n_0$ and $m,$ $\mathbf{c}$ in $\mathbb{Z}[t]^{n_0}$, $A$ in $\mathbb{Z}^{m \times n_0}$, and $\mathbf{b}$ in $(\mathbb{Z}t + \mathbb{Z})^m$ which defines regions $R(t)$ and $L(t)$ and optimum value function $f_{\ell_0}.$ 

Here is the general approach. We first prove that for all $\ell,$ to find $f_{\ell}(t),$ it suffices to look at only points which are less than $\ell$ away from a bounding hyperplane of $R(t)$ and ignore all other points. This motivates us to construct a parametric set $L^*(t)$ which contains exactly the points in $L(t)$ which are less than $\ell$ away from a bounding hyperplane of $R(t).$ This is most naturally achieved by constructing a finite number of PILPs in the same $n_0$ indeterminates such that in each one, the feasible points lie in a hyperplane parallel and close to a bounding hyperplane of $R(t)$. Integrality ensures that for each bounding hyperplane of $R(t),$ the points less than $\ell$ away lie on finitely many parallel hyperplanes. Note that a single bounding hyperplane of $t$ does not change its normal vector because $A$ has constant entries. We carefully construct the new PILPs so that they are mutually exclusive; this allows us to characterize the $\ell^{\text{th}}$ optimum for $\ell > 1.$  Finally, we transform a PILP whose feasible set lies on an affine subspace into a PILP with one less indeterminate by using standard manipulations of lattices.

\begin{proposition} \label{ctb} For all $t$ and $\ell,$ $f_{\ell}(t)$ equals the $\ell^{\text{th}}$ largest value of $\mathbf{c}^{\intercal}(t) \mathbf{x}$ over all $\mathbf{x}$ in $L(t)$ that are less than $\ell$ away from some bounding hyperplane of $R(t)$ (in the Euclidean metric).
\end{proposition}

\begin{proof} Fix $t$ and $\ell.$ 

\textbf{Case 1:} $|L(t)|<\ell.$ Then $f_{\ell}(t)=-\infty,$ and the $\ell^{\text{th}}$ largest value of $\mathbf{c}^{\intercal}(t) \mathbf{x}$ over all $\mathbf{x}$ in $L(t)$ that are less than $\ell$ away from some bounding hyperplane of $R(t)$ is naturally equal to $-\infty.$ See Proposition \ref{ht} for specifics.

\textbf{Case 2:} $|L(t)| \ge \ell.$ Then $f_{\ell}(t) \neq -\infty.$

\textbf{Case 2.1:} $\mathbf{c}^{\intercal}(t) = \mathbf{0}.$ The objective function is constant. To prove the proposition in this case, we need to show that at least $\ell$ distinct points in $L(t)$ are less than $\ell$ away from some bounding hyperplane of $R(t).$ Successively choose $\ell$ points $\mathbf{v_1}, \ldots, \mathbf{v_\ell}$ with minimal $x_1$ coordinate from $L(t)$ without replacement, which is possible since $|L(t)| \ge \ell.$ Let $\mathbf{e_1}$ be the unit vector in the first coordinate. For each $\mathbf{v_i},$ $\mathbf{v_i}-\mathbf{e_1}, \ldots, \mathbf{v_i}-\ell \mathbf{e_1}$ cannot all be in $L(t),$ or else $\mathbf{v_i}$ could not be picked in this process. Therefore, for $i=1, \ldots, \ell,$ $\mathbf{v_i}$ is less than $\ell$ away from the boundary of $R(t),$ as desired.

\textbf{Case 2.2:} $\mathbf{c}^{\intercal}(t) \neq \mathbf{0}.$ To prove the proposition, it suffices to show that for $k=1, \ldots, \ell,$ the $k^{\text{th}}$ largest value does not occur at any point of $L(t)$ which is at least $\ell$ away from the boundary of $R(t).$ Suppose, for the sake of contradiction, that for $k \le \ell,$ the $k^{\text{th}}$ largest value occurs at a point $\mathbf{v}$ in $L(t)$ which is at least $\ell$ away from the boundary of $R(t).$ Since the objective function is linear but not constant, it is greater at some lattice point 1 away from $\mathbf{v},$ say $\mathbf{w},$ than it is at $\mathbf{v}.$ By assumption, for $s=1, \ldots, \ell,$ $\mathbf{v}+s(\mathbf{w}-\mathbf{v})$ lies in $R(t)$ hence in $L(t),$ and 
\begin{displaymath} \mathbf{c}^{\intercal}(t) (\mathbf{v}+s(\mathbf{w}-\mathbf{v})) > \mathbf{c}^{\intercal}(t) (\mathbf{v}), \end{displaymath} which contradicts that the $k^{\text{th}}$ largest value occurs at $\mathbf{v}.$ Therefore, the assumption was wrong, as desired.
\end{proof}

 By standard vector calculations, a lattice point less than $\ell_0$ away from the hyperplane $\sum_{i=1}^{n_0} a_i x_i = a_0$ lies on a hyperplane $\sum_{i=1}^{n_0} a_i x_i = a_0 + j,$ where \begin{displaymath}|j| < \ell_0 \sqrt{\sum_{i=1}^n a_i^2}.\end{displaymath}

If $a_0, \ldots, a_n$ are all integers, then $j$ is an integer (and there are finitely many hyperplanes to consider).

We can write the inequalities $\mathbf{x} \ge 0$ as extra rows for $A$ and $\mathbf{b}$ in an obvious way. The constraints of $Q$ are now $A \mathbf{x} \le \mathbf{b}(t).$ We ignore the rows of $A$ which are all $0$ because these rows define inequalities that are either true for sufficiently large $t$ or false for sufficiently large $t.$ In the former case, we can ignore the row for $t \gg 0,$ and in the latter, $f_{\ell_0}(t)$ is eventually $-\infty$. We redefine $m$ so that $A$ has $m$ rows. 

For all $i$ with $1 \le i \le m,$ let
\begin{displaymath} c_i := \left \lceil \ell_0 \sqrt{\sum_{j=1}^n A_{ij}^2} \; \; \right \rceil. \end{displaymath}

We define $c_i$ PILPs which correspond to the parametric hyperplanes near to and parallel to the parametric hyperplane $(A\mathbf{x})_i \le \mathbf{b}(t)_i.$ For all such $i$ and $k=0, 1, \ldots, c_i-1$ let $Q_{i, k}$ be the PILP in canonical form which is our original PILP with the same objective function and constraints with some extra constraints: 
\begin{align} \label{qik_hyperplane} \left( \sum_{j=1}^{n_0} A_{ij} x_j \ge b_i(t) - k  \right) \wedge \left( - \sum_{j=1}^{n_0} A_{ij} x_j \ge - b_i(t) + k\right) \end{align}
and for $h=1, \ldots, i-1,$ 
\begin{align} \label{qik_exclude} \sum_{j=1}^{n_0} A_{hj} x_j \le b_h(t) - c_h.  \end{align}

 Note that $b_i$ and $b_h$ have degree at most $1.$ Let $Q_{i, k}$ define regions $R_{i, k}(t)$ and $L_{i, k}(t)$ and optimum value functions $\{f_{\ell, i, k}\}.$ (As before, it is possible for many of the sets $L_{i, k}(t)$ to be empty.) The condition $($\ref{qik_hyperplane}$)$ restricts the PILP to a single parametric hyperplane. The conditions in $($\ref{qik_exclude}$)$ are imposed in order to make the PILPs mutually exclusive. In other words, for fixed $t,$ the sets $L_{i, k}(t)$ are disjoint.

\begin{proposition} \label{ht} For all $t,$ $f_{\ell_0}(t)$ equals the $\ell_0^{\text{th}}$ largest value in the multiset 
\begin{displaymath} \{f_{\ell, i, k}(t) \vert 1 \le \ell \le \ell_0, 1 \le i \le m, 0 \le k < c_i\}. \end{displaymath}
\end{proposition}

\begin{proof} 
By construction, for fixed $t,$ the regions $L_{i, k}(t)$ are disjoint. For fixed $t,$ consider $\mathbf{x}$ in $L(t)$ that is less than $\ell_0$ away from the boundary of $R(t).$ There are finitely many bounding hyperplanes, so there exists a minimum $i$ such that $\mathbf{x}$ is less than $\ell_0$ away from the hyperplane given by the $i^{\text{th}}$ row. Then $\mathbf{x}$ lies in $L_{i, k}(t),$ where $k$ is the distance from $\mathbf{x}$ to the $i^{\text{th}}$ hyperplane (at $t$) times $(\sum_j A_{ij}^2)^{1/2}$ (this product is less than $c_i$). This shows that the disjoint union of $L_{i, k}(t)$ for $1 \le i \le m$ and $0 \le k < c_i$ is the set of all points in $L(t)$ that are less than $\ell_0$ away from a bounding hyperplane of $R(t).$ Combining this fact with Proposition \ref{ctb} shows that $f_{\ell_0}(t)$ is the $\ell_0^{\text{th}}$ largest value in the union of the multisets of the $\ell_0$ largest values from all $L_{i, k}(t).$
\end{proof}

\begin{proposition} \label{p5} Assume that $\ell_0, n_0$ are positive integers such that Theorem \ref{RCF} is true for all positive integers $\ell, n$ such that $\ell + n < \ell_0 + n_0.$ Fix $\ell, i,$ and $k$ such that $\ell \le \ell_0,$ $1 \le i \le m,$ $0 \le k < c_i.$ $f_{\ell, i, k}$ is EQP. \end{proposition}

\begin{proof} The idea is that since $R_{i, k}(t)$ lies on a hyperplane, $Q_{i, k}$ is like some PILP with $n_0-1$ indeterminates. Intuitively, $f_{\ell, i, k}$ should be closely related to the $\ell^{\text{th}}$ optimum value function for this new PILP. 

For all $t,$ $R_{i, k}(t)$ lies on the hyperplane $\sum_{j=1}^{n_0} A_{i,j} x_j = b_i(t) - k;$ call this hyperplane $P(t).$ Let $d=\gcd(A_{i,1}, \ldots, A_{i, n_0}),$ which can be written as $\sum_{h=1}^{n_0} \beta_h A_{i, h}$ for integers $\beta_h.$ Then $P(t)$ contains lattice points (not necessarily in $L_{i, k}(t)$) if and only if $b_i(t)-k$ is a multiple of $d.$ Since $b_i$ has degree at most $1,$ this occurs either never or periodically; in particular there exist positive integers $p$ and $q$ such that $P(t)$ contains lattice points if and only if $t \equiv p \pmod{q}.$ In the former case (``never"), $f_{\ell, i, k}(t)=-\infty,$ and we are done. Therefore, assume the ``periodic" case. Then $f_{\ell, i, k}(t)=-\infty$ when $t \not \equiv p \pmod{q}.$

Suppose that $t \equiv p \pmod{q}.$ Let $g(t)$ be the polynomial $\frac{1}{d} (b_i(t)-k),$ which has degree at most $1$ and is an integer for these $t$. The lattice point $(g(t)\beta_1, \ldots, g(t)\beta_{n_0})$ lies in $P(t).$ The lattice points in $P(t)$ translated by $-(g(t)\beta_1, \ldots, g(t)\beta_{n_0})$ form an integer lattice $U$ which is independent of $t$ and which spans an $n_0-1$ dimensional space. Therefore, there exist independent vectors in $\mathbb{Z}^{n_0},$ $\mathbf{e_1}, \ldots, \mathbf{e_{n_0-1}},$ such that
$\varphi_t: \mathbb{Z}^{n_0-1} \rightarrow \mathbb{Z}^{n_0}$ given by 

\begin{displaymath} \varphi_t(\mathbf{y}) = (g(t)\beta_1, \ldots, g(t)\beta_{n_0}) + \sum_{h=1}^{n-1} y_h\mathbf{e_h} \end{displaymath}
is injective with image $\mathbb{Z}^{n_0} \cap P(t).$

Define a new PILP $Q'$ \textit{not} in canonical form with indeterminates $\mathbf{y}=(y_1, \ldots, y_{n_0-1})$ (which correspond to coefficients of the $\mathbf{e_h}$), regions $R'(t)$ and $L'(t),$ and optimum value functions $\{f_\ell'\}.$ Each constraint of $Q_{i, k}$ has the form $\mathbf{w}^{\intercal} \mathbf{x} \le W(t),$ where $\mathbf{w}$ is an integer vector and $W$ is in $\mathbb{Z} t + \mathbb{Z}$. For each such constraint, $Q'$ has the constraint
\begin{displaymath} \mathbf{w}^{\intercal} \left((g(t) \beta_1, \ldots, g(t) \beta_{n_0}) + \sum_{h=1}^{n_0-1} y_h \mathbf{e_h}\right) \le W(t), \end{displaymath}
and $Q'$ has no other constraints. Since $W$ and $g$ are rational polynomials of degree at most $1,$ the constraints of $Q'$ can be equivalently written as parametric inequalities of the integer indetermates $\mathbf{y}$ where the coefficients on the right hand side are in $\mathbb{Z} t + \mathbb{Z}$ and the coefficients on the left hand side are integers. The objective function, which we denote by $c(\mathbf{y}),$ is
\begin{displaymath} \mathbf{c}^{\intercal}(t) \left((g(t)\beta_1, \ldots, g(t)\beta_{n_0}) + \sum_{h=1}^{n-1} y_h\mathbf{e_h}\right), \end{displaymath}
which can be written as a rational polynomial covector times $\mathbf{y}$ plus a rational polynomial. 

By construction, for each $t$ which is $p \pmod{q},$ the map $\varphi_t$ defined above maps $R'(t)$ to $R_{i, k}(t).$ The map $\varphi_t$ is a bijection between these two sets because $\varphi_t$ is a bijection between $\mathbb{Z}^{n_0-1}$ and $\mathbb{Z}^{n_0} \cap P(t).$ The bijection commutes with evaluating the respective objective function, so $f_{\ell, i, k}(t) = f'_\ell(t).$ (When $t \not \equiv p \pmod{q},$ recall that $f_{\ell, i, k}(t) = -\infty.$) If the PILP $Q'$ satisfied the hypotheses of Theorem \ref{RCF}, then we would be done. This is not the case, so we have a bit more work to do. In particular, the lattice point set of $Q'$ might not lie in the first orthant, and its objective function might not have the form of an integer polynomial covector times the vector of indeterminates.

For all $t,$ $R'(t)$ is bounded because it is contained in some degenerate affine transformation of $R_{i, k}(t)$. Using the same argument as in Proposition \ref{bounded}, vertices of $R'(t)$ have coordinates bounded in magnitude by $(n_0-1)! \alpha(t)^{n_0 - 2} \beta(t),$\footnote{If $n_0 = 1,$ then instead, observe that $Q'$ has zero indeterminates, so it is already in the first orthant, so let $Q'' = Q'.$} which is in turn bounded by a linear function of $t.$ We are using that the constraints of $Q'$ are parametric inequalities in which the coefficients on the right hand side have degree at most $1$ and the coefficients on the left hand side are integers. Let $K$ be an integer so that $Kt + K$ is greater than this bound. Let $Q''$ be a new PILP in which the constraints of $Q'$ are ``translated'' by $(Kt+K, \ldots, Kt+K).$ To be precise, this means that the lattice point set of $Q''$ equals that of $Q'$ translated by $(Kt+K, \ldots, Kt+K)$ so that the former lattice point set lies in the first orthant. Let $Q''$ have the exact same objective function as $Q',$ $c(\mathbf{y}).$ Note that the constraints of $Q''$ are also parametric inequalities whose right hand sides have degree at most $1$ and whose left hand sides are integers.


The function $c(\mathbf{y})$ does not have the form of an integer polynomial covector times the vector of indeterminates, but there exist $z$ in $\mathbb{Z}_+$ and $Z$ in $\mathbb{Z}[t]$ such that $zc(\mathbf{y}) - Z(t)$ is an integer polynomial covector times $\mathbf{y}.$ Let $Q'''$ be yet another PILP with the same indeterminates and constraints with objective function $zc-Z(t),$ which is multiplication by an integer polynomial covector. Let $\{f''_\ell\}$ be the family of optimum value functions for $Q'',$ etc.


$R'''(t)$ is bounded for all $t,$ and by the definition of $K,$ $R'''(t)$ lies in the first orthant, so we can write $Q'''$ in reduced canonical form by adding the constraints $\mathbf{y} \ge \mathbf{0}$ without changing its optimum value functions. The PILP $Q'''$ has $n_0-1$ indeterminates, the right hand sides have degree at most $1,$ and $\ell \le \ell_0.$ We are assuming that Theorem \ref{RCF} applies when either $(\ell + n < \ell_0+n_0$ and $n > 0)$ or $n = 0,$ and one of these cases holds for $Q''',$ so $f'''_\ell$ is EQP. 

For $t$ such that $f'''_\ell(t)$ is finite, $zf''_\ell(t)-Z(t)=f'''_\ell(t).$ For $t$ such that $f'''_\ell(t)=-\infty,$ $f''_\ell(t)=-\infty.$ 
This easily implies that $f''_\ell$ is EQP. 

As in the end of Section \ref{S2}, the translation from $Q'$ to $Q''$ adds a polynomial to the optimum value functions or takes $-\infty$ to $-\infty.$ Therefore, $f'_\ell$ is EQP. Restricting an EQP to the residue class $p \pmod{q}$ and replacing the other outputs by $-\infty$ gives another EQP, so $f_{\ell, i, k}$ is EQP, as desired.
\end{proof}

\begin{proof}[Proof of Theorem \ref{RCF}] We have already established the $n=0$ case. We use strong induction on $\ell+n.$ Let $\ell_0, n_0$ be positive integers and assume that Theorem \ref{RCF} is true for all positive integers $\ell, n$ such that $\ell + n < \ell_0 + n_0.$ There is no base case of $\ell_0+n_0$ because the case $n=0$ acts as the base case.

By Proposition \ref{ht}, $f_{\ell_0}(t)$ equals the $\ell_0^{\text{th}}$ largest value among the multiset \begin{displaymath} \{f_{\ell, i, k}(t) \vert 1 \le \ell \le \ell_0, 1 \le i \le m, 0 \le k < c_i\}. \end{displaymath} By Proposition \ref{p5}, $f_{\ell, i, k}$ is EQP for $\ell \le \ell_0, 1 \le i \le m,$ and $0 \le k < c_i.$ Proposition \ref{max} shows that $f_{\ell_0}$ is EQP, which shows that Theorem \ref{RCF} is true for the case $\ell=\ell_0$ and $n=n_0,$ which completes the induction.
\end{proof}

\begin{remark} \label{two-variables} Theorems \ref{GF} and \ref{SF} are not true for more than one parameter. For example, consider the following PILP with positive integer parameters $t_1, t_2$ and indeterminates $a, b, c.$ The constraints are $0 < a, b, c < 2 t_1 t_2$ and $a t_1 + c = b t_2.$ The objective function is to minimize $c.$ The optimum is easily seen to be
$ \min(c) = \gcd(t_1, t_2).$ This two-parameter function is not considered EQP.


In Theorems \ref{GF} and \ref{SF}, one cannot replace $\ell$ with a polynomial. For example, consider the PILP with indeterminates $x_1$ and $x_2$ in which $R(t)$ is the region $x_1 \ge 0, x_2 \ge 0, x_1 + x_2 \le t$ and $\mathbf{c}=(1,0).$ Then $f_t(t)-t \sim \sqrt{2t}$ because, roughly speaking, there are about $t$ points in the top $\sqrt{2t}$ rows of $L(t).$ The function $t \mapsto f_t(t)-t$ is not EQP, being asymptotic to $\sqrt{2t},$ so neither is $f.$ \end{remark} 

\section{The convex hull of the lattice point set}
In this section, we study the convex hull of the lattice point set of a PILP and resolve a conjecture by Calegari and Walker. This convex hull is closely related to the first optimum value function of this PILP because the optimum value of a linear objective function occurs at a vertex of this convex hull.

In 2011, Calegari and Walker proved the following as Theorem 3.5 in \cite{CW}. 

\begin{theorem} \label{CW} For $1 \le i \le k,$ let $\mathbf{v_i}$ be in $\mathbb{Q}(t)^n$ of size $O(t),$ and for all $t,$ let $R(t)$ be the convex hull of $\mathbf{v_1}(t), \ldots, \mathbf{v_k}(t).$ Then there exists a positive integer $d$ such that for $t \gg 0$ and $t$ restricted to a single residue class $\pmod{d},$ there exist $\mathbf{w_1}, \ldots, \mathbf{w_z}$ in $\mathbb{Q}[t]^n$ such that the convex hull of the set of lattice points in $R(t)$ has vertices $\mathbf{w_1}(t), \ldots, \mathbf{w_z}(t).$
\end{theorem}

The number of vertices of the convex hull of the lattice point set, $z,$ need not be the same for each residue class or even positive. They proved this theorem by showing that the vertices are of bounded distance from the bounding hyperplanes as $n$ goes to $\infty,$ which is similar to our argument in Section 3. They conjectured that this theorem is still true for parametric vertices $\mathbf{v_i}$ which are not of size $O(t)$ (see 3.9 in their paper), which we now prove. 

\begin{theorem} \label {CW2} For $1 \le i \le k,$ let $\mathbf{v_i}$ be in $\mathbb{Q}(t)^n,$ and for all $t,$ let $R(t)$ be the convex hull of $\mathbf{v_1}(t), \ldots, \mathbf{v_k}(t).$ Then there exists a positive integer $d$ such that for $t \gg 0$ and $t$ restricted to a single residue class $\pmod{d},$ there exist $\mathbf{w_1}, \ldots, \mathbf{w_z}$ in $\mathbb{Q}[t]^n$ such that the convex hull of the set of lattice points in $R(t)$ has vertices $\mathbf{w_1}(t), \ldots, \mathbf{w_z}(t).$ \end{theorem}

Here is the general approach. We first rephrase this problem in the language of a PILP, $Q$, in standard form. Proposition \ref{bt1} then applies and relates this PILP with the disjoint union of finitely many PILPs, $Q_{\alpha}$, in reduced canonical form via a map $\varphi_t,$ a map whose inverse is affine. Theorem \ref{CW} is shown to apply to PILPs in reduced canonical form, so the convex hull of the $L_{\alpha}(t)$ have EQP structure. The next main step is to prove that the convex hull of $L(t)$ is related in a certain way to $\varphi^{-1}$ applied to the convex hulls of $L_{\alpha}(t).$ This reduces the task of proving Theorem \ref{CW2} to a question purely about the convex hull of a finite set of polynomial vectors, Proposition \ref{ch3}.

\begin{proposition} To prove Theorem \ref{CW2}, it suffices to show that the vertices of the convex hull of the lattice point set of a PILP in standard form has eventually quasi-polynomial structure. (See the second sentence of Theorem \ref{CW2}.)\end{proposition}

\begin{proof} Let $\mathbf{v_1}, \ldots, \mathbf{v_k}$ be in $\mathbb{Q}(t)^n.$ Because rational functions are bounded by polynomials for $t \gg 0$, we can translate these parametric vertices by some polynomial vector so that $R^*(t)$ lies in the first orthant for $t \gg 0.$ One can show that $R(t)$ coincides with the real vector set of some PILP in canonical form. See Section 2.1 of \cite{CLS}. Since $R^*(t)$ and $R(t)$ are translations by a polynomial vector, it suffices to show that the vertices of the convex hull of the lattice point set of a PILP in canonical form has eventually quasi-polynomial structure.

The bijection between the lattice point sets of a PILP in canonical form and a PILP in standard form from Section $\ref{S2}$ extends to a bijective affine transformation (the codomain is the hyperplane of the PILP in standard form). Bijective affine transformations preserve convex combinations and send polynomial vectors to polynomial vectors. Therefore, the vertices of the lattice point sets of two PILPs are in bijection by this same affine transformation, and it suffices to consider a PILP in standard form. \end{proof}

Consider a PILP in standard form, $Q$, with regions $R(t)$ and $L(t).$ There exists $r$ such that $R(t)$ bounded in magnitude by $t^r$ for $t \gg 0.$ Define PILPs $Q_{\alpha},$ etc. as in Proposition \ref{bt1}. In this proposition, we showed that for $t \gg 0,$ the map 
\begin{displaymath} \varphi_t: L(t) \rightarrow \sqcup_{\alpha \in S} L_{\alpha}(t) \end{displaymath} 
given by $\varphi_t(\mathbf{x}) = \mathbf{y}$ with $x_i = \sum_{j=1}^r y_{i, j} t^{j-1}$ is a bijection. 


Let $M(t)$ be the vertices of the convex hull of $L(t)$, and define $M_{\alpha}(t)$ similarly. 

\begin{proposition} \label{chs} Fix $t \gg 0$ so that $\varphi_t$ is a bijection. The image of $M(t)$ under $\varphi_t$ lies in $\sqcup_{\alpha} M_{\alpha}(t).$ \end{proposition}

\begin{proof}
Suppose that a point $p$ in $\sqcup_{\alpha} L_{\alpha}(t)$ is a convex combination of other points in $\sqcup_{\alpha} L_{\alpha}(t),$ say $\sum_i c_i p_i$ where for all $i,$ $c_i \ge 0,$ $p \neq p_i \in \sqcup_{\alpha} L_{\alpha}(t),$ and $\sum_i c_i = 1.$ It is easy to see that $\varphi_t^{-1}$ preserves this convex combination:
\begin{displaymath} \varphi_t^{-1}(p)= \sum_i c_i \varphi_t^{-1}(p_i). \end{displaymath}
A point in $L(t)$ is not a convex combination of other points in $L(t)$ if and only if it is in $M(t)$. By the contrapositive of the above observation, the image under $\varphi_t$ of an element of $M(t)$ is a vertex of the convex hull of $\sqcup_{\alpha} L_{\alpha}(t),$ so it is in one of the sets $M_{\alpha}(t).$
\end{proof}

To understand $M_{\alpha}(t),$ we wish to apply Theorem \ref{CW} to $R_{\alpha}(t).$ Each bounding hyperplane of $R_{\alpha}(t)$ has the form $\mathbf{a}^{\intercal} \mathbf{y} = b(t),$ where $\mathbf{a}$ is in $\mathbb{Z}^{rn}$ and $b$ has degree at most 1. As in Proposition \ref{bounded}, for all $t,$ each vertex of $R_{\alpha}(t)$ is the intersection of $rn$ of the bounding parametric hyperplanes (at $t$) which intersect at a single point. A size $rn$ subset of the parametric hyperplanes has unique intersection if and only if their left hand sides form an invertible matrix, and this is independent of $t.$ Their intersection is the inverse of this matrix times the vector of the right hand side. Since the matrix has integer entries, the intersection has the form $\mathbf{v_{1, A}} t + \mathbf{v_{2, A}}$ where $\mathbf{v_{1, A}}, \mathbf{v_{2, A}}$ are in $\mathbb{Q}^{rn}$ i.e. is of size $O(t).$

\begin{proposition} \label{cand} Fix $t.$ A candidate intersection $\mathbf{v_{1, A}} t + \mathbf{v_{2, A}}$ is a vertex of $R_{\alpha}(t)$ if and only if it satisfies all of the parametric inequalities at $t.$ \end{proposition}

\begin{proof} If a candidate intersection $\mathbf{v_{1, A}} t + \mathbf{v_{2, A}}$ is a vertex of $R_{\alpha}(t),$ it lies in $R_{\alpha}(t),$ so it satisfies all of the parametric inequalities. 

Conversely, suppose that a candidate intersection is not a vertex of $R_{\alpha}(t).$ One case is that it lies in $R_{\alpha}(t)$. We use a standard characterization of vertices of a polytope: there exists a vector $\mathbf{r}$ such that $\mathbf{v_{1, A}} t + \mathbf{v_{2, A}} + s \mathbf{r}$ lies in $R_{\alpha}(t)$ for all reals $s$ in some neighborhood of $0.$

By construction, $\mathbf{v_{1, A}} t + \mathbf{v_{2, A}}$ lies on $rn$ hyperplanes which meet at exactly one point (and possibly other hyperplanes). Therefore, at least one of these hyperplanes is not fixed by addition by $\mathbf{r}.$ By definition, $R_{\alpha}(t)$ excludes all of the points on one half space determined by this hyperplane, which contradicts the previous paragraph. Therefore, the candidate does not lie in $R_{\alpha}(t),$ and the candidate does not satisfy all of the parametric inequalities, as desired.
\end{proof}

\begin{proposition} \label{CW3} Theorem \ref{CW} applies to $R_{\alpha}(t).$ \end{proposition}

\begin{proof} Proposition \ref{cand} gives us a description of the vertices of $R_{\alpha}(t).$ A parametric inequality is the comparison of two (linear) polynomials, and the comparison of two polynomials stabilizes for $t \gg 0.$ Since there are finitely many parametric inequalities and candidate intersections, for $t \gg 0,$ all of the comparisons stabilise, and each candidate is either eventually a vertex of $R_{\alpha}(t)$ or eventually not a vertex. 

For all $t,$ $R_{\alpha}(t)$ is convex, so for $t \gg 0,$ $R_{\alpha}(t)$ is the convex hull of the candidates which are eventually vertices. Each candidate is of size $O(t),$ so Theorem \ref{CW} applies. \end{proof}

\begin{proposition} \label{ch2} Fix $t \gg 0$ so that $\varphi_t$ is a bijection. Let
\begin{displaymath}T(t) := \bigcup_{\alpha \in S} \varphi_t^{-1}(M_{\alpha}(t)),\end{displaymath}
the union of the preimages of $M_{\alpha}(t).$ For $t \gg 0,$ the vertices of the convex hull of $T(t)$ is exactly $M(t).$ \end{proposition}

\begin{proof}

By Proposition \ref{chs}, the image of $M(t)$ under $\varphi_t$ lies in $\sqcup_{\alpha} M_{\alpha}(t)$, so each element of $M(t)$ is in $T(t).$ Each element of $T(t)$ is an element of $L(t)$ because its image under $\varphi_t$ lies in one of the sets $L_{\alpha}(t).$  Each element of $M(t)$ is a vertex of the convex hull of $L(t)$ and lies in $T(t),$ so it is a vertex of the convex hull of $T(t).$ 

Assume, for the sake of contradiction, that some vector $\mathbf{v}$ is a vertex of the convex hull of $T(t)$ but not in $M(t).$ Then $\mathbf{v}$ is in $L(t)$ but not in $M(t)$, so it is a convex combination of the elements of $M(t).$ Because $M(t)$ is a subset of $T(t),$ $\mathbf{v}$ in $T(t)$ is a convex combination of other elements of $T(t),$ a contradiction. Therefore, for $t \gg 0,$ the vertices of the convex hull of $T(t)$ is exactly $M(t).$ 
\end{proof}

By Proposition \ref{CW3}, for each $\alpha,$ there is a positive integer $d$ such that when $t$ is restricted to a residue class $\text{mod }d,$ the number of vertices of the convex hull of $L_{\alpha}(t)$ is eventually constant, and they are eventually given by polynomial vectors. Since there are finitely many $\alpha,$ we can find $d$ that satisfies this property for all $\alpha.$ 

Restrict $t$ to a single residue class $\text{mod }d.$ By Theorem \ref{CW}, for $t \gg 0,$ we can write $M_{\alpha}(t)$ as $\{\mathbf{y_{\alpha, \beta}} (t)  \mid \beta \in S_{\alpha}\}$ where $S_{\alpha}$ is finite and $\mathbf{y_{\alpha, \beta}}$ are distinct vectors in $\mathbb{R}[t]^n.$ Then for $t \gg 0$ (and $t$ restricted to the residue class), $T(t)$ is a finite set of polynomial vectors.

By Proposition \ref{ch2}, for $t \gg 0,$ the convex hull of $T(t)$ is exactly $M(t).$ It turns out that we no longer need the details of $Q.$ Theorem \ref{CW2} can be proven with the following proposition.

\begin{proposition} \label{ch3} Let $\mathbf{w_1}, \ldots, \mathbf{w_p}$ be distinct elements of $\mathbb{R}[t]^n.$ There exists a subset $U$ of $\{1, \ldots, p\}$ such that for $t \gg 0,$ $\{\mathbf{w_h}(t) \mid h \in U \}$ is the set of vertices of the convex hull of $\mathbf{w_1}(t), \ldots, \mathbf{w_p}(t)$ with no repeats. \end{proposition}
\begin{corollary} Theorem \ref{CW2}\end{corollary}
\begin{proof} We have reduced the proof of Theorem \ref{CW2} to showing that the vertices of the convex hull of the lattice point set of a PILP in standard form has EQP structure. For $t \gg 0,$ $M(t)$ equals the convex hull of $T(t).$ For $t$ restricted to a residue class $\text{mod }d$ and $t \gg 0,$ there is a finite set of polynomial vectors that gives $T(t).$ We can eliminate the identical polynomial vectors. By Proposition \ref{ch3}, for $t \gg 0$ (and $t$ restricted to this residue class), the set of vertices of the convex hull of $T(t)$ is a fixed subset of polynomial vectors. This is true for all single residue classes $\text{mod }d,$ so $M(t)$ has the desired EQP structure $\text{mod }d.$
\end{proof}

 We need a few facts to prove Proposition \ref{ch3}.

\begin{theorem}[{Carath\'eodory's Theorem \cite{Car}}]
A point that lies in the convex hull of a subset $P$ of $\mathbb{R}^d$ lies in a simplex with vertices in $P.$ \end{theorem}

\begin{proposition}\label{ch31} Let $V$ be any subset of $\{1, \ldots, p\}.$ Suppose that for some positive integer $t_0,$ the set $\{\mathbf{w_h}(t_0) \mid h \in V\}$ is affinely independent. Then $\{\mathbf{w_h}(t) \mid h \in V\}$ is affinely independent for $t \gg 0.$ \end{proposition}

\begin{proof} Let $X$ be the matrix of polynomials with columns which are $\mathbf{w_h}-\mathbf{w_{\min(V)}}$ for all $h$ in $V \setminus \{\min(V)\}.$ For all $t,$ $\{\mathbf{w_h}(t) \mid h \in V\}$ is affinely dependent if and only if the determinants of all $|V|-1 \times |V|-1$ minors of $X$ evaluated at $t$ equal $0.$ There exists such a minor of $X$ whose determinant is nonzero at $t_0.$ The determinant is a polynomial of $t$, so it is nonzero at $t$ for $t \gg 0,$ which gives the desired result. \end{proof}

\begin{proposition} \label{ch32}  Let $V$ be a subset of $\{1, \ldots, p\}.$ Suppose that the set $\{\mathbf{w_h}(t) \mid h \in V\}$ is affinely independent for $t \gg 0.$ Let $\mathbf{w}$ be in $\mathbb{R}[t]^n.$ Then either $\mathbf{w}(t)$ is a convex combination of $\{\mathbf{w_h}(t) \mid h \in V\}$ for $t \gg 0$ or $\mathbf{w}(t)$ is not a convex combination of $\{\mathbf{w_h}(t) \mid h \in V\}$ for $t \gg 0.$
\end{proposition}
This is still true if the vectors are not affinely independent for $t \gg 0,$ but it is harder to prove.
\begin{proof} By Proposition \ref{ch31}, either $\{\mathbf{w_h}(t) \mid h \in V\} \cup \{\mathbf{w}(t)\}$ is affinely independent for $t \gg 0$ or it is affinely dependent for $t \gg 0.$ In the former case, $\mathbf{w}$ is not a convex combination of $\{\mathbf{w_h}(t) \mid h \in V\}$ for $t \gg 0,$ as desired. 

Now consider the latter case. The set $\{\mathbf{w_h}(t) \mid h \in V\}$ is affinely independent, and  $\\ \{\mathbf{w_h}(t) \mid h \in V\} \cup \{\mathbf{w}(t)\}$ is affinely dependent for $t \gg 0.$ From standard linear algebra, there exist functions $c_h(t)$ for $h$ in $V$ such that for $t \gg 0,$ $\sum_{h \in V} c_h(t)=1$ and 
\begin{align} \label{E4} \mathbf{w}(t) = \sum_{h \in V} c_h(t) \mathbf{w_h}(t).\end{align}

This expresses $\mathbf{w}(t)$ as an affine combination of $\{\mathbf{w_h}(t) \mid h \in V\}.$ For $t \gg 0$, $c_h(t)$ is uniquely determined. We claim that the functions $c_h$ eventually agree with rational functions. Define $X$ as in Proposition \ref{ch31}. There is a $|V|-1 \times |V|-1$ minor of $X,$ call it $X^*,$ whose determinant is a nonzero polynomial. Let $P$ be the projection onto the corresponding $|V|-1$ rows (the corresponding $|V|-1$ dimensional subspace of $\mathbb{R}^d$). Then $\{P(\mathbf{w_h}(t)) \mid h \in V\}$ is affinely independent for $t \gg 0,$ and $P(\mathbf{w}(t))$ is an affine combination of $\{P(\mathbf{w_h}(t)) \mid h \in V\}.$ As before, there exist functions $c^*_h(t)$ for $h$ in $V$ such that $\sum_{h \in V} c^*_h(t)=1$ and
\begin{displaymath} P(\mathbf{w}(t)) = \sum_{h \in V} c^*_h(t) P(\mathbf{w_h})(t), \end{displaymath}
which are uniquely defined for $t \gg 0.$ An equivalent equation is 
\begin{displaymath} P(\mathbf{w}(t)-\mathbf{w_{\min(V)}}(t)) = \sum_{h \in V \setminus \min(V)} c^*_h(t) P(\mathbf{w_h}-\mathbf{w_{\min(V)}}(t)). \end{displaymath}

The functions $c^*_h(t)$ for $h \in V \setminus \min(V)$ is given by multiplying the inverse of the matrix associated with the right hand side by the vector of the left hand side. Specifically the vector with coordinates $c^*_h(t)$ for $h \in V \setminus \min(V)$ is given by multiplying the inverse of $X^*$ by the left hand side. Therefore, $c^*_h$ for $h \in V \setminus \min(V)$ is a rational function (more precisely, agrees with a rational function) for $t \gg 0$. Since $\sum_{h \in V} c^*_h(t)=1,$ $c^*_{\min(V)}$ is also a rational function. On the other hand, applying the projection to (\ref{E4}) gives
\begin{displaymath} P(\mathbf{w}(t)) = \sum_{h \in V} c_h(t) P(\mathbf{w_h})(t), \end{displaymath}
so for $h$ in $V$ and $t \gg 0,$ $c_h(t)=c_h^*(t).$ Therefore, the functions $c_h$ eventually agree with rational functions.

For $t \gg 0,$ $\mathbf{w}(t)$ is a convex combination of $\{\mathbf{w_h}(t) \mid h \in V\}$ if and only if for all $h$ in $V,$ $0 \le c_h(t) \le 1.$ A rational function either eventually lies in $[0,1]$ or eventually does not lie in $[0,1].$  Thus, either $\mathbf{w}(t)$ is a convex combination of $\{\mathbf{w_h}(t) \mid h \in V\}$ for $t \gg 0$ or $\mathbf{w}(t)$ is not a convex combination of $\{\mathbf{w_h}(t) \mid h \in V\}$ for $t \gg 0,$ as desired. \end{proof}

\begin{proof}[Proof of Proposition \ref{ch3}]
 Let $i$ be in $\{1, \ldots, p\}.$ For all $t,$ $\mathbf{w_i}(t)$ is a vertex of the convex hull of $\mathbf{w_1}(t), \ldots, \mathbf{w_p}(t)$ if and only if it is not a convex combination of $\{\mathbf{w_h}(t)\}_{h \neq i}.$ By Carath\'eodory's Theorem, $\mathbf{w_i}(t)$ is a convex combination of $\{\mathbf{w_h}(t)\}_{h \neq i}$ if and only if it lies in a simplex with vertices in this set. By Proposition \ref{ch31}, each subset of $\{\mathbf{w_h}\}_{h \neq i}$ is either a simplex for $t \gg 0$ or not a simplex for $t \gg 0.$ There are finitely many subsets, so for $t \gg 0,$ these properties stabilize. By Proposition \ref{ch32}, for each subset of $\{\mathbf{w_h}\}_{h \neq i}$ which is a simplex for $t \gg 0,$ $\mathbf{w_i}(t)$ is either a convex combination for $t \gg 0$ or not a convex combination for $t \gg 0.$ Combining these facts tells us that either $\mathbf{w_i}(t)$ is a vertex of the convex hull of $\mathbf{w_1}(t), \ldots, \mathbf{w_p}(t)$ for $t \gg 0$ or not a vertex of the convex hull for $t \gg 0.$

Let $U$ be the set of all $i$ such that $\mathbf{w_i}(t)$ is a vertex of the convex hull of $\mathbf{w_1}(t), \ldots, \mathbf{w_p}(t)$ for $t \gg 0.$ Since $\mathbf{w_1}, \ldots, \mathbf{w_p}$ are distinct polynomial vectors, for $t \gg 0,$ $\{\mathbf{w_h}(t) \mid h \in U \}$ is the set of vertices of the convex hull of $\mathbf{w_1}(t), \ldots, \mathbf{w_p}(t)$ with no repeats, as desired.
\end{proof}

\begin{remark} Theorem \ref{CW} may be used to prove Theorem \ref{RCF} using induction on $\ell.$ Consider a PILP in reduced canonical form. The vertices of $R(t)$ are of size $O(t),$ so Theorem \ref{CW} applies. Let $d$ be prescribed by the theorem. For $\ell=1$ and $t$ restricted to a residue class $\pmod d,$ $f_\ell(t)$ is the maximum over the vertices of the convex hull of the lattice point set, which is a finite set of polynomial vectors, of $\mathbf{c}^{\intercal}(t),$ which is then a maximum of polynomials. One of these polynomials is maximal for $t \gg 0$; it corresponds a single polynomial vector $\mathbf{v}$ at which the maximum is obtained for $t\gg 0.$ (There may exist others.) Since $f_\ell$ is EQP when restricted to each residue class $\pmod{d},$  $f_\ell$ is EQP.

Suppose that Theorem \ref{RCF} is true for $\ell=\ell_0$ and consider $\ell=\ell_0+1.$ Let $Q$ be a PILP in reduced canonical form, $d$ as in Theorem \ref{CW}, and $t$ restricted to a residue class $\pmod{d}$ (throughout this paragraph). The maximum value of $\mathbf{c}^{\intercal}(t)$ is obtained at some polynomial vector $\mathbf{v}$ in $L(t).$ Suppose that there is another PILP $Q'$ for which $L'(t)=L(t) \setminus \{\mathbf{v}(t)\}$ (for $t$ restricted to the residue class $\pmod{d}$). Then $f_\ell(t)=f'_{\ell-1}(t).$ Intuitively, it is possible to construct $Q'$ by adding one more constraint to $Q$ to exclude $\mathbf{v}(t)$ because it is a vertex of the convex hull of $L(t)$ (for $t$ restricted to the residue class $\pmod{d}$). We do not do so explicitly here. By assumption, $f'_{\ell-1}$ is EQP, so $f_\ell$ is EQP when restricted to each residue class $\pmod{d},$ so we are done. \end{remark}

\begin{remark} The above idea easily proves that the $\ell^{\text{th}}$ largest value without multiplicity is eventually quasi-polynomial. Again, we use induction on $\ell.$ The case $\ell=1$ is as before. The inductive step now only involves adding the constraint $\mathbf{c}^{\intercal}(t) \mathbf{x} \le \mathbf{c}^{\intercal}(t) \mathbf{v}(t)-1.$ \end{remark}

\section*{Acknowledgements}

The author thanks Joe Gallian, Levent Alpoge, and an anonymous referee for helpful comments on the manuscript and advice to simplify proofs. This research was conducted at the University of Minnesota Duluth REU and was supported by NSF grant 1358695 and NSA grant H98230-13-1-0273.

\end{document}